\documentclass[10pt]{amsart}

\def\beq*{\begin{eqnarray*}}
\def\eeq*{\end{eqnarray*}}

\def\b{\beta}
\def\a{\alpha}
\def\haa{\widehat{a}}
\def\hbb{\widehat{b}}
\def\n{\noindent}
\def\hj{\widehat{j}}
\def\hn{\widehat n}
\def\mbP{\mathbb{P}}

\def\hb{\widehat{\beta}}
\def\ha{\widehat{\alpha}}
\def\n{\noindent}
\def\mcS{\mathcal{S}}

\usepackage[usenames,dvipsnames]{xcolor}

\newtheorem{case}{Case}
\newtheorem{subcase}{Case}
\numberwithin{subcase}{case}

\newtheorem{theorem}{Theorem}
\newtheorem{lemma}[theorem]{Lemma}
\newtheorem{corollary}[theorem]{Corollary}

\newtheorem{definition}{Definition}
\newtheorem{remark}{Remark}

\author[Pawe{\l} Hitczenko]{Pawe{\l} Hitczenko${}^\dagger$}
\thanks{$\dagger$ Partially supported by a grant from 
Simons Foundation (grant \#208766 to Pawe{\l} Hitczenko)}
\address{Department of Mathematics, Drexel University, Philadelphia, 
PA  19104, USA} 
\email{phitczenko@math.drexel.edu}

\author[Amanda Lohss]{Amanda Lohss${}^\dagger$}
\address{Department of Mathematics, Drexel University, Philadelphia, 
PA  19104, USA} 
\email{agp47@drexel.edu}

\title[Distribution of parameters in staircase tableaux]{On the asymptotic  distribution of
 parameters
\\ 
in  random weighted  staircase tableaux
}

\begin{document}
\maketitle
\begin{abstract}
 In this paper, we study staircase tableaux, a combinatorial object
 introduced due to its connections with the asymmetric
 exclusion process (ASEP) and Askey-Wilson polynomials. Due to their
 interesting connections, staircase tableaux have been the object of
 study in many recent papers. More specific to this paper, the
 distribution of various parameters in random staircase tableaux has
 been studied. There have been interesting results on parameters
 along the main diagonal, however, no such results have appeared for
 other diagonals. It was conjectured that the distribution of the
 number of symbols
 along the $k$th diagonal is asymptotically Poisson as $k$ and the
 size of the tableau tend to infinity.  We partially prove this
 conjecture; more specifically we prove it for the second and the third main diagonal.
\end{abstract}

\section{Introduction}
\label{sec:in}
In this paper, we study staircase tableaux, a combinatorial object
introduced (in \cite{CW1}, \cite{CW2}) due to connections with the asymmetric exclusion
process (ASEP)  and Askey-Wilson polynomials. The
ASEP can be defined as a Markov chain with $n$ sites, with at most
one particle occupying each site. Particles may jump to any
neighboring empty site with rate $u$ to the right and rate $q$ to the
left. Particles may enter and exit at the first site with rates $\a$ and $\gamma$ respectively. Similarly, particles may enter and exit the last site with rates $\delta$ and $\b$. The ASEP is an interesting particle model that has been
studied extensively in mathematics and physics. It has also been studied in many other fields, including computational biology \cite{Bu}, and biochemistry, specifically as a primitive model for protein synthesis \cite{GMP}.
Staircase tableaux were introduced per a connection between the steady state distribution of the ASEP and the generating function for staircase tableaux \cite{CW2}. See Section \ref{DN} for a discussion of the ASEP and its connection with staircase tableaux.

In addition to interest in its own right, the ASEP has been known to
have interesting connections in combinatorics and
analysis. Consequently, staircase tableaux have similarly been
connected to many combinatorial objects and a family of
polynomials. In fact, the generating function for staircase tableaux
has been used to give a formula for the moments of Askey-Wilson
polynomials  \cite{CW2}, \cite{CSSW}. Staircase tableaux have also
inherited many interesting properties from other types of tableaux
(See \cite{ABN}, \cite{CH}, \cite{CN}, \cite{CW2}, \cite{CW4}, \cite{CW3},
\cite{HJ3}, \cite{SW}). We refer to \cite[Table~1]{HJ} for a
description of some of the bijections between the various types of the
tableaux. \\

Due to all these interesting connections, staircase tableaux have been the
object of study in many recent papers. 
 In particular, the asymptotic distribution of parameters along the main diagonal is
known. The number of $\alpha/\gamma$ symbols and the number of
$\beta/\delta$ symbols along the main diagonal were proven to be
asymptotically normal in \cite{D-HH}, and the distribution of boxes
along the main diagonal was given in \cite{HJ}. However, the
distributions of parameters on the other diagonals have not been
studied specifically.

In \cite{HJ}, the
distribution of each box in a staircase tableau was given, and  
the expected values of the number of symbols on the diagonals were computed.
 Based on that  it was
conjectured that the distribution of the number of symbols along the $k$th diagonal
is asymptotically Poisson as $k$ and the size of the tableau  tend to infinity. The main results of this
paper is the proof of this in two special cases when $k=n-1$ and $k=n-2$. That is, we show
that the distribution of the number of symbols along the second and the third main diagonal is
asymptotically Poisson with parameter $1$ (see
Theorem~\ref{thm:symb} and Theorem~\ref{thm:3rdx} below). Similarly, we show that the number of $\a$'s
(resp. $\b$'s) along the second and the third main diagonal is
asymptotically Poisson with parameter $1/2$ (see
Theorem~\ref{DT}, Corollary~\ref{BC}, and Theorem~\ref{alpha_on_3rd}). 

The paper is organized as follows. We  introduce terminology, notation and provide some background information in the next section. Section~\ref{SAB} concerns the second main diagonal and the last section is devoted to the third main diagonal. The results for the second main diagonal were discussed in an extended abstract~\cite{HP}.

\section{Notation and Preliminaries } \label{DN}
Staircase tableaux were first introduced in \cite{CW1} and \cite{CW2} as follows:
\begin{definition} \label{DEFST}
A staircase tableau of size n is a Young diagram of shape (n, n-1, ..., 1) such that:
\begin{enumerate}
\item The boxes are empty or contain an $\a$, $\b$, $\gamma$, or $\delta$.
\item All boxes in the same column and above an $\a$ or $\gamma$ are empty.
\item All boxes in the same row and to the left of an $\beta$ or $\delta$ are empty.
\item Every box on the diagonal contains a symbol.
\end{enumerate}
\end{definition}
The rows and columns in a staircase tableau are numbered from $1$ through $n$, beginning with the box in the NW-corner and continuing south and east respectively. Each box is numbered $(i,j)$ where $i,j \in \{ 1, 2, ..., n \}$. Note that $i + j \leq n + 1$. We refer to the collection of boxes $(n-i+1, i)$ such that $i=1,2, ..., n$ as the main diagonal,  the collection of boxes $(n-i, i)$ such that $i=1,2, ..., n-1$ as the second main diagonal, and  the collection of boxes $(n-i-1, i)$ such that $i=1,2, ..., n-2$ as the third main diagonal.

Following the conventions of \cite{HJ}, $\mcS_{n}$ is the set of all staircase tableaux of size $n$. For a given $S \in \mcS_{n}$, the number of $\a$'s, $\gamma$'s, $\b$'s and $\delta$'s in $S$ are denoted by $N_{\a}, N_{\b}, N_{\gamma},$ and $N_{\delta}$ respectively. The weight of $S$ is the product of all symbols in $S$:
\[
wt(S) = \a^{N_{\a}}\b^{N_{\b}}\gamma^{N_{\gamma}}\delta^{N_{\delta}}.
\]
It was known (see e.g. \cite{CD-H}) that the generating function $Z_{n}(\alpha, \beta, \gamma, \delta) := \sum_{S \in \mcS_{n}} wt(S)$ is equal to the product:
\begin{equation}\label{Z4}
Z_{n}(\a, \b, \gamma, \delta)=\prod^{n-1}_{i=0} (\a + \b + \delta + \gamma + i(\a + \gamma)(\b + \delta)).
\end{equation}

\vspace{1cm}

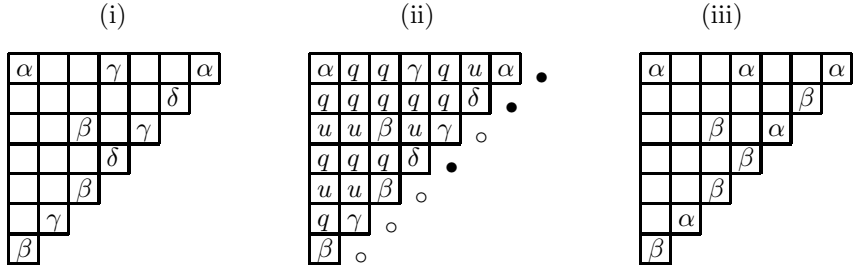
\begin{figure}[htbp] 
\setlength{\unitlength}{0.4cm}
\begin{center}
\begin{picture}
(0,0)(15,8)\thicklines

\put(3,8){(i)} \put(13,8){(ii)}\put(23,8){(iii)}


\put(0,0){\line(0,1){7}}
\put(1,0){\line(0,1){7}}
\put(2,1){\line(0,1){6}}
\put(3,2){\line(0,1){5}}
\put(4,3){\line(0,1){4}}
\put(5,4){\line(0,1){3}}
\put(6,5){\line(0,1){2}}
\put(7,6){\line(0,1){1}}

\put(0,7){\line(1,0){7}}
\put(0,6){\line(1,0){7}}
\put(0,5){\line(1,0){6}}
\put(0,4){\line(1,0){5}}
\put(0,3){\line(1,0){4}}
\put(0,2){\line(1,0){3}}
\put(0,1){\line(1,0){2}}
\put(0,0){\line(1,0){1}}

\put(3.25,6.25){$\gamma$}
\put(6.25, 6.25){$\a$}
\put(5.25, 5.25){$\delta$}
\put(2.25, 4.25){$\b$}
\put(4.25, 4.25){$\gamma$}
\put(3.25, 3.25){$\delta$}
\put(2.25, 2.25){$\b$}
\put(1.25, 1.25){$\gamma$}
\put(0.25, 0.25){$\b$}
\put(0.25, 6.25){$\a$}

\put(10,0){\line(0,1){7}}
\put(11,0){\line(0,1){7}}
\put(12,1){\line(0,1){6}}
\put(13,2){\line(0,1){5}}
\put(14,3){\line(0,1){4}}
\put(15,4){\line(0,1){3}}
\put(16,5){\line(0,1){2}}
\put(17,6){\line(0,1){1}}

\put(10,7){\line(1,0){7}}
\put(10,6){\line(1,0){7}}
\put(10,5){\line(1,0){6}}
\put(10,4){\line(1,0){5}}
\put(10,3){\line(1,0){4}}
\put(10,2){\line(1,0){3}}
\put(10,1){\line(1,0){2}}
\put(10,0){\line(1,0){1}}

\put(16.25, 6.25){$\a$}
\put(15.25, 6.25){$u$}
\put(14.25, 6.25){$q$}
\put(13.25, 6.25){$\gamma$}
\put(12.25, 6.25){$q$}
\put(11.25, 6.25){$q$}
\put(10.25, 6.25){$\a$}
\put(15.25, 5.25){$\delta$}
\put(14.25, 5.25){$q$}
\put(13.25, 5.25){$q$}
\put(12.25, 5.25){$q$}
\put(11.25, 5.25){$q$}
\put(10.25, 5.25){$q$}
\put(10.25, 4.25){$u$}
\put(11.25, 4.25){$u$}
\put(12.25, 4.25){$\b$}
\put(13.25, 4.25){$u$}
\put(14.25, 4.25){$\gamma$}
\put(10.25, 3.25){$q$}
\put(11.25, 3.25){$q$}
\put(12.25, 3.25){$q$}
\put(13.25, 3.25){$\delta$}
\put(10.25, 2.25){$u$}
\put(11.25, 2.25){$u$}
\put(12.25, 2.25){$\beta$}
\put(10.25, 1.25){$q$}
\put(11.25, 1.25){$\gamma$}
\put(10.25, 0.25){$\b$}

\put(17.5,6.0){$\bullet$}
\put(16.5,5.0){$\bullet$}
\put(15.5,4.0){$\circ$}
\put(14.5,3.0){$\bullet$}
\put(13.5,2.0){$\circ$}
\put(12.5,1.0){$\circ$}
\put(11.5,0.0){$\circ$}

\put(21,0){\line(0,1){7}}
\put(22,0){\line(0,1){7}}
\put(23,1){\line(0,1){6}}
\put(24,2){\line(0,1){5}}
\put(25,3){\line(0,1){4}}
\put(26,4){\line(0,1){3}}
\put(27,5){\line(0,1){2}}
\put(28,6){\line(0,1){1}}

\put(21,7){\line(1,0){7}}
\put(21,6){\line(1,0){7}}
\put(21,5){\line(1,0){6}}
\put(21,4){\line(1,0){5}}
\put(21,3){\line(1,0){4}}
\put(21,2){\line(1,0){3}}
\put(21,1){\line(1,0){2}}
\put(21,0){\line(1,0){1}}

\put(24.25,6.25){$\a$}
\put(27.25, 6.25){$\a$}
\put(26.25, 5.25){$\b$}
\put(23.25, 4.25){$\b$}
\put(25.25, 4.25){$\a$}
\put(24.25, 3.25){$\b$}
\put(23.25, 2.25){$\b$}
\put(22.25, 1.25){$\a$}
\put(21.25, 0.25){$\b$}
\put(21.25, 6.25){$\a$}

\end{picture}

\vspace{3cm}
\caption{A staircase tableau of size $7$ with weight $\a^{2}
  \b^{3} \delta^{2} \gamma^{3}$. (ii) The extension of (i) to a
  staircase tableau of weight $\a^{2} \b^{3} \delta^{2}
  \gamma^{3} u^{6} q^{13}$ 
and type $\bullet \bullet\circ  \bullet \circ \circ \circ$. 
(iii) The $\a/\b$-staircase tableau obtained from  (i) by replacing $\gamma$'s with $\a$'s and $\delta$'s with $\b$'s. It's weight is $\a^{5}\b^{5}$. }
\label{F1}
\end{center}

\end{figure}

Notice that there is an involution on the staircase tableaux of a
given size obtained by interchanging the rows and the columns, $\a$'s and $\b$'s, and $\gamma$'s and $\delta$'s, see further \cite{CD-H}. In particular, the fact that $\a$'s and $\b$'s are identical up to this involution allows us to extend results for $\a$'s to results for $\b$'s.

The connection between staircase tableaux and the ASEP requires an
extension of this preceding definition. After following the rules from
Definition \ref{DEFST}, we then fill all the empty boxes with $u$'s
and $q$'s, the rates at which particles in the ASEP jump to the right
and left respectively. We do so by first filling all boxes to the left
of a $\b$ with a $u$ and to the left of a $\delta$ with a $q$. Then,
we fill the empty boxes with a $u$ if it is above an $\a$ or a $\delta$, and $q$ otherwise. The weight of a staircase tableau filled as such is defined in the same way, the product of the parameters in each box. Also, the total weight of all such staircase tableaux, which we denote by $\mcS^{'}_{n}$, is given by:
\[
Z_{n}(\a, \b, \gamma, \delta, q, u) := \sum_{S \in \mcS^{'}_{n}} wt(S).
\]
Then, each staircase tableau of size $n$ is associated with a state of
the ASEP with $n$ sites (See Figure \ref{F1}). This is done by
aligning the Markov chain with the diagonal entries of the staircase tableau. A site is filled if the corresponding diagonal entry is an $\a$ or a $\gamma$ and a site is empty if the corresponding diagonal entry is a $\b$ or a $\delta$. Each staircase tableau's associated state of the ASEP is called its type.

Using this association, it was shown in \cite{CW2} that the steady state probability that the ASEP is in state $\eta$ is given by:
\[
\frac{\sum_{T \in \mathfrak{T}} wt(T)}{Z_{n}}, 
\]
where $\mathfrak{T}$ is the set of all staircase tableau of type $\eta$.

For the purposes of this paper, we will consider 
more simplified
staircase tableaux, namely $\a/\b$-staircase tableaux as
introduced in \cite{HJ}, which are staircase tableaux limited to the
symbols $\a$ and $\b$. The set $\overline{\mcS}_{n} \subset
\mcS_{n}$ denotes the set of all such staircase tableaux. Since the symbols $\a$ and $\gamma$ follow the same rules in the definition, as do $\b$ and $\delta$, any $S \in \mcS_{n}$ can be obtained from an $S^{'} \in \overline{\mcS}_{n}$ by replacing the appropriate $\a$'s with $\gamma$'s and $\b$'s with $\delta$'s. 

The generating function of $\a/\b$-staircase tableaux is:
\[
Z_{n}(\a, \b) := \sum_{S \in \overline{\mcS}_{n}} wt(S) = Z_{n}(\a, \b, 0, 0)
\]
and it follows from $(\ref{Z4})$ that it 
is simply:
\[
Z_{n}(\a, \b) = \a^{n}\b^{n}(a + b)^{\overline{n}}= \a^{n}\b^{n}(a + b+n-1)_{n}
\]
where $a:= \a^{-1}$ and $b:=\b^{-1}$, a notation that will be used frequently throughout this paper,  $(x)^{\overline{n}}$ is the rising factorial of $x$, i.e. $(x)^{\overline{n}} = x(x+1)\cdots(x+n-1)$, and $(x)_n=x(x-1)\dots(x-(n-1))$ is the falling factorial of $x$.

We wish to consider random staircase tableaux as was done in
\cite{D-HH} but it suffices to study random $\a/\b$-staircase
tableaux 
as was done in
\cite{HJ}. All of our results for random $\a/\b$-staircase
tableaux can be extended to random staircase tableaux with all four
parameters, $\a, \gamma, \b, \delta$. This is done by
randomly
replacing each $\a$ with $\gamma$ with probability
$\frac{\gamma}{\a + \gamma}$ and similarly, each $\b$ with
$\delta$ with probability $\frac{\delta}{\b + \delta}$
independently for each occurrence. Notice that $Z_{n}(\a, \b, \gamma, \delta) = Z_{n}(\a + \gamma, \b + \delta)$. We also allow all parameters to be arbitrary positive real numbers,  i.e. $\a, \b \in (0, \infty)$, allowing $\a = \infty$ by fixing $\b$ and taking the limit or vice versa, or $\a=\b=\infty$ by taking the limit. Several such cases were considered in the literature and we refer to \cite[Section~3]{HJ} for examples and discussion. 

The following is a formal definition of a weighted random staircase  introduced  in \cite{HJ}, although we prefer to use a different notation:

\begin{definition}
For all $n \geq 1$, $\a, \b \in [0, \infty)$ with $(\a,
\b) \neq (0,0)$, we consider a family of probability measures $\mbP_{n,\a,\b}$  on $\overline{\mcS}_{n}$ defined by:
\[
\mathbb{P}_{n,\a,\b} (S) = \frac{wt(S)}{Z_{n}(\a,\b)} = \frac{\a^{N_{\a}}\b^{N_{\b}}}{Z_{n}(\a,\b)},\quad S \in \overline{\mcS}.
\]
We denote by $\overline{\mcS}_{n,\a,\b}$ the probability space $(\overline\mcS_n,\mathbb{P}_{n,\a,\b})$ and we call  $S\in\overline\mcS_{n,\a,\b}$
 a random weighted staircase tableau (with weights $\a$, $\b$). That is, 
$S$ is an
$\a/\b$-staircase tableau in $\overline{\mcS}_{n}$ chosen according to the 
 probability distribution $\mbP_{n,\a,\b}$.
\end{definition}
As in  \cite{HJ}  
 $\a$ and $\b$ are used  in two different meanings, as fixed symbols in the tableaux  and as the values of the parameters, but that should not cause any confusion.

We identify the boxes of a tableau by a row and column number (as in a matrix) and we write $y_{i,j}$ to indicate that a box $(i,j)$  contains a symbol $y$, where $y$ is $\a$, $\b$, or $0$.
Using the above definition, Hitczenko and Janson presented the distribution
of a given box in a random staircase tableau.  
If a box is  on the main diagonal, its distribution, in our notation,  is (see
\cite[Theorem~7.1]{HJ}):
\begin{equation}\label{1BOXD}
\mathbb{P}_{n, \a, \b}(\a_{i, n+ 1 - i}) = \frac{n-i+b}{n + a + b - 1},\quad \mathbb{P}_{n, \a, \b}(\b_{i, n+ 1 - i}) = \frac{a+i-1}{n + a + b - 1}.
\end{equation}
Since a box on the main diagonal is never empty,  
the second formula follows trivially from the first. This is no longer true for  boxes not on the main diagonal. For such boxes the distribution of non--zero symbols  is (see
\cite[Theorem~7.2]{HJ}):
\begin{eqnarray}  && 
\label{1BOXA}
\mathbb{P}_{n, \a, \b}(\a_{i, j}) = \frac{j - 1 + b}{(i + j + a + b - 1)_2
}, 
\qquad\mathbb{P}_{n, \a, \b}(\b_{i, j})  = \frac{i - 1 + a}{(i + j + a + b - 1)_2
}, 
\end{eqnarray}
and the remainder is for $\mbP_{n,\a,\b}(0_{i,j})$.

For an arbitrary $S \in \overline{\mcS}_{n}$ and an arbitrary box $(i,
j)$ in S, define $S[i,j]$ to be the  subtableau in
$\overline{\mcS}_{n-i-j+2}$ obtained by deleting the first $i-1$ rows
and $j-1$ columns, see \cite{HJ}. 
The following statement was proven in \cite[Theorem~6.1]{HJ} and is a useful tool in our results: If 
$S\in\overline\mcS_{n,\a,\b}$ then $S[i,j]\in\overline\mcS_{n-i-j+2,\ha,\hb}$  with  $\widehat{a} = a+i-1$  and  $\widehat{b}=b+j-1$. This means that   if $\widetilde S\in\overline{\mcS}_{n-i-j+2}$ then 
\begin{equation} \label{SUBT}
\mbP_{n,\alpha, \beta}(S\in\overline{\mcS}_n:\ S[i, j] =\widetilde S)=\mbP_{n-i-j+2, \widehat \a, \widehat \b}(\widetilde S), 
\end{equation}

To show  the convergence to a Poisson random variable we will rely on the method of (factorial) moments.  When the random variables are the sums of indicators the condition takes the following form
\begin{lemma}\label{fac_mom} Let $Y=\sum_{j=1}^mI_j$, where $(I_j)$ are indicator random variables.  Then, for $r\ge1$, 
\[\mathbb{E}(Y)_{r} =
r! \left( \sum_{1 \leq j_{1} < ... < j_{r} \leq m } \mathbb{P}(I_{j_{1}} \cap \ldots \cap I_{j_{r}}) \right), 
\]
where $\mathbb E(X)_r=\mathbb EX(X-1)\dots(X-(r-1))$ is the $r^{th}$  factorial moment.
\end{lemma}
\begin{proof}
We have
\begin{eqnarray*} 
z^Y &=& z^{\sum^{m}_{j=1} I_j} = \prod^{m}_{j=1} z^{I_j} = \prod^{m}_{j=1} (1 + (z-1))^{I_ j} = \prod^{m}_{j=1}(1 + I_j(z-1)) \\&=&
1 + \sum^{m}_{r=1} \left( \sum_{1 \leq j_{1} < ... < j_{r} \leq m} \left( \prod^{r}_{k=1} I_{j_k} \right) \right) (z - 1)^{r} \\&=& 1 + \sum^{m}_{r=1} (z - 1)^{r} \left( \sum_{1 \leq j_{1} < ... < j_{r} \leq m} \left( \prod^{r}_{k=1} I_{j_{k}} \right) \right).
\end{eqnarray*}
Thus, 
\[
\mathbb{E}(z^{Y}) = 1 + \sum^{m}_{r=1} (z - 1)^{r} \left( \sum_{1 \leq
  j_{1} < ... < j_{r} \leq m} 
\mathbb{P}(I_{j_{1}} \cap \ldots \cap
I_{j_{r}}) \right). 
\]
Hence
\[\mathbb{E}(Y)_{r} =
\frac{d^r}{dz^r}(\mathbb{E}z^{Y}) |_{z=1} =
r! \left( \sum_{1 \leq j_{1} < ... < j_{r} \leq m} \mathbb{P}(I_{j_{1}} \cap \ldots \cap I_{j_{r}}) \right). 
\]
\end{proof} 
Thus, to apply the above lemma we will need to find the  joint distribution of symbols in given boxes on the diagonal of a tableau. We will obtain the exact formulas for the second main diagonal (see Theorems~\ref{rSymbols}   and \ref{rNZero} in the next section) and asymptotic for the third main diagonal (see Lemmas~\ref{rsymbols} and  \ref{rsymbolsx} below). It should be noted that the joint distribution of symbols for the main diagonal was given in \cite[Theorem~7.6]{HJ}, but since boxes on the main diagonal  cannot be empty it is a different issue.

\section{Distribution of parameters along the second main diagonal} \label{SAB}

The following two lemmas consider the probability of an arbitrary
staircase tableau in $\overline{\mcS}_n$ that is conditioned on having an $\alpha$ or a $\beta$ in the box $(n-1, 1)$. The statements follow almost immediately from the definition of a staircase tableau, but will be used frequently throughout the paper. 
\begin{lemma}
\label{SWCorner}If $S\in\overline\mcS_{n,\a,\b}$  is conditioned on having $\a$ in box $(n-1,1)$, then 
 the subtableau $S_{n, \a, \b} [1, 3]\in\overline\mcS_{n-2, \a, \b}$. In other words, for $S\in\overline\mcS_{n,\a,\b}$
 \[
\mbP_{n,\a,\b}(S\ |\ \a_{n-1,1})=\mbP_{n-2,\a,\b}(S[1,3]).\]
\end{lemma}
\begin{proof}
If $S\in\overline \mcS_{n}$ is a staircase tableau such that  $\a_{n-1, 1}$, then  $\beta_{n, 1}$ and $\a_{n-1, 2}$
by the rules of a staircase tableau. The first and second column are otherwise empty by those same rules. The remainder, $S[1, 3]$,  is an arbitrary staircase tableau of size $n-2$. Therefore, the lemma follows.
\end{proof}

\begin{lemma}
\label{BSWCorner} Let $(S)_{i,j}$ be a subtableau of $S$
with the $i$th row and the $j$th column removed. 
If $S\in\overline\mcS_{n, \alpha, \beta}$ is conditioned on having $\beta$ in box $(n-1,
1)$, then the subtableau $(S)_{n-1, 2}$ 
 is random tableau in $\overline\mcS_{n-1,\a,\b}$  conditioned on having a $\b$ in the
$(n-1, 1)$ box. In other words
\[\mbP_{n,\a,\b}(S\ |\
\b_{n-1,1})=\mbP_{n-1,\a,\b}((S)_{n-1,2}\ |\ \b_{n-1,1}).\]
\end{lemma}
\begin{proof}
If $S\in\overline{\mcS}_{n}$ is a staircase tableau such that  $\beta_{n-1, 1}$, then  $\a_{n-1, 2}$ and $\beta_{n,1} $ by the rules of a
staircase tableau. The second column is otherwise empty by those same
rules. The $n$th row  only has one box, $(n,1)$, which must be a
$\beta$. The remainder is an arbitrary staircase tableau of size $n-1$
conditioned to have a $\beta$ in box $(n-1,1)$. Therefore, the lemma follows.
\end{proof}
We are now ready to state and prove the results for the second main diagonal. We begin with the distribution of the number
of $\a$'s and we treat the number of non--empty boxes in the next subsection.

\subsection{The number of $\a$'s on the second main diagonal}
As our first result, the following is the distribution of boxes along the second main diagonal. 
In order to simplify notation, let $\a_j$  be the event $\a_{n-j,j}$ i.e. that  a box $(n-j,j)$ on the second diagonal and in the $j^{th}$ column contains an  $\alpha$.  

\begin{theorem}
\label{rSymbols}
Let $ 1 \leq j_{1} < ... < j_{r} \leq n - 1$. If 
\begin{equation}\label{2_diff}j_{k} \leq j_{k + 1} - 2, \hspace{2mm} \forall k = 1, 2, ..., r -
1\end{equation} 
then 
\begin{eqnarray*}&&\mathbb{P}_{n, \alpha, \beta}(\a_{j_1},\dots,\a_{j_{r}}) 
= \prod_{k=1}^{r}\frac{b + j_{r - k + 1} - 2r + 2k - 1}{(n+a+b - 2r + 2k -1)_2
}.\end{eqnarray*}
\noindent (For $r=1$, this is (\ref{1BOXA})). Otherwise, 
\[\mathbb{P}_{n, \alpha, \beta}(\a_{j_1},\dots, \a_{j_r})  = 0.\] 
\end{theorem}

\begin{proof} First note that when (\ref{2_diff}) fails there
 exists $j_{k}$ such that $j_{k} = j_{k+1} - 1$ and thus there must be two $\alpha$'s  in boxes side by side on the $(n-i, i)$ 
diagonal. But this is impossible by the rules of a staircase
tableau as  no symbol  can be put in the diagonal box
$(n-j_k,j_{k+1})$ adjacent to these two boxes. Therefore the
probability is $0$. 

Suppose now that (\ref{2_diff}) holds. We proceed by induction on $r$. Set
\begin{equation}\label{hats}
\hb^{-1} = \beta^{-1} + j_{1} - 1,\quad \hn:=n-(j_1-1),\quad \hj_l:=j_l-(j_1-1), \quad 1\le l\le r.\end{equation}
Then by \eqref{SUBT}
\[
\mbP_{n,\a,\b}(\a_{j_1},\dots,\a_{j_r})=\mbP_{\hn,\a,\hb}(\a_{\hj_1},\dots,\a_{\hj_r})
=\mathbb{P}_{\hn , \alpha, \hb}(\a_{\hj_2}, \dots,\alpha_{\hj_{r} }|\a_{\hj_1}) \cdot
\mathbb{P}_{\hn, \alpha, \hb}(\alpha_{\hj_1}).
\]
Since $\hj_1=1$, by Lemma \ref{SWCorner} and the induction hypothesis (applied with $\widetilde n:=\hn-2$, $\widetilde j_l:=\hj_{l+1}-2$, $1\le l\le r-1$)
\begin{eqnarray*} &&
\mathbb{P}_{\hn , \alpha, \hb}(\a_{\hj_2} ,\dots,\alpha_{\hj_{r}} \hspace{1mm} | \hspace{1mm} \ha_{j_1}) 
=
\mathbb{P}_{\hn - 2, \alpha, \hb}(\a_{\hj_{2} -2} ,\dots,\a_{\hj_{r} - 2})  \\&&\quad=
\prod_{k=1}^{r-1} \frac{\widehat{b} + \widetilde j_{(r-1)-k+1} - 2(r-1) + 2k-1}{(\widetilde n+a + \hbb  - 2(r-1) + 2k-1)_2}
 =
\prod_{k=1}^{r-1} \frac{b + j_{r-k+1} - 2r+ 2k-1}{(n + a + {b}  - 2r + 2k-1)_2},
\end{eqnarray*}
where, in the last step we used \eqref{hats}. 
By (\ref{1BOXA}) and \eqref{hats}:
\[
\mathbb{P}_{\hn , \alpha, \hb}(\alpha_{1}) = \frac{\hbb}{(\hn + a + \hbb - 1)_2}=\frac{b+j_1-1}{(n+a+b-1)_2}.
\]
Therefore,
\begin{eqnarray*} 
\mathbb{P}_{n, \alpha, \beta}(\a_{j_1},\dots,\alpha_{j_r}) 
&=&
\prod_{k=1}^{r} \frac{b + j_{r-k+1} - 2r + 2k - 1}{(n+a + b - 2r + 2k - 1)_2
}
\end{eqnarray*}
which proves the result.
\end{proof}

Our second main result of this section is the distribution of the number of $\alpha$'s
(and $\beta$'s) along the second main diagonal. The proof requires a  lemma.
\begin{lemma}
\label{LA}
Let 
\[J_{r,m}:=\{ 1 \leq j_{1} < ... < j_{r} \leq m:\ j_{k} \leq j_{k + 1}
- 2, \hspace{2mm} \forall k = 1, 2, ..., r - 1\}.\] 
Then 
\[\sum_{J_{r, m}} \left(\prod_{k=1}^{r} j_{r-k+1}\right)
=\frac{(m+1)_{2r}}{2^{r}r!}.\]
\end{lemma}
\begin{proof}
By induction on $r$.
When $r=1$:
\[\sum_{J_{1,m}} \left( \prod_{k=1}^{1} j_{1-k+1} \right) = \sum_{j_1=1}^m j_{1} = \frac{(m+1)m}{2}.
\]
Assume the statement holds for $r-1$. Then:
\begin{eqnarray*}
\sum_{J_{r,m}} \left(\prod_{k=1}^{r} j_{r-k+1} \right) &=&
\sum_{j_r=2r-1}^{m} j_r\left( \sum_{J_{r-1,j_r-2}}  \prod_{k=2}^{r} j_{r-k+1} \right) \\&=&
\sum_{j_r=2r-1}^{m} j_{r}
\frac{(j_r-1)_{2(r-1)}}{2^{r-1}(r-1)!}\\&=&\frac1{2^{r-1}(r-1)!}\sum_{j_{r}=2r-1}^m
( j_{r})_{2r-1}
\end{eqnarray*}
where the second equality is by the induction hypothesis.
Since
\[\sum_{j_r=2r-1}^{m} (j_{r})_{2r-1}=\sum_{j_{r}=0}^m (
j_{r})_{2r-1}\]
the lemma will be proved once we verify that 
\[\sum_{j=0}^m(j)_t
=\frac{(m+1)_{t+1}}{t+1},\]
for any non-negative integer $t$ (and apply it with $t=2r-1$). 
Using the identity 
\[\sum_{j=0}^m{j\choose t}={m+1\choose t+1}\]
(see, e.g. \cite[Formula~(5.10)]{GKP})   we see that 
\[\sum_{j=0}^m(j)_t=
\sum_{j=0}^{m} \frac{j!}{(j-t)!}=t!\sum_{j=0}^m{j\choose t}
=t!{m+1\choose  t+1}
=t!\frac{(m+1)_{t+1}}{(t+1)!}=\frac{(m+1)_{t+1}}{m+1},\]
as asserted.
 \end{proof}

Finally, define $A_{n}$ and $B_{n}$ to be the number of $\alpha$'s and
$\beta$'s on the second main diagonal, 
i.e. 
$A_{n} := \sum^{n-1}_{j=1} I_{\alpha_j}$ and
$B_{n} := \sum^{n-1}_{j=1} I_{\beta_j}$, where  $\b_j$ means $\b_{n-j,j}$. Then,
the asymptotic distribution of $A_n$ and $B_n$ is given in the following theorem and corollary.
\begin{theorem}
\label{DT}
Let $Pois(\lambda)$ be a Poisson random variable with parameter
$\lambda$. Then, as $n\to\infty$,
\begin{equation}
A_{n} \stackrel{d}{\rightarrow} Pois \left( \frac{1}{2} \right).
\end{equation}
\end{theorem}
\begin{proof}
By \cite[Theorem~20, Chapter~1]{B} it suffices to show that the $r$th
factorial moment of 
$A_{n}$ 
satisfies:
\begin{equation}
\mathbb{E}(A_{n})_{r}\to\left(\frac12\right)^r\hspace{5mm}\mbox{as\ }\n\to\infty.
\end{equation}
By Lemma~\ref{fac_mom}, Theorem~\ref{rSymbols}, and Lemma~\ref{LA}
\begin{eqnarray*} 
\mathbb{E}(A_{n})_{r}& =& 
r! 
\sum_{J_{r,n-1}} \left( \prod_{k=1}^{r} \frac{b + j_{r-k+1} - 2r + 2k - 1}{(n+a+b - 2r + 2k - 1)_2
} \right) 
\approx
r!\sum_{J_{r, n - 1}} \left( \prod_{k=1}^{r} \frac{j_{r-k+1}}{n^{2}}
\right) \\&=& \frac{r!}{n^{2r}} \frac{(n)_{2r}}{2^{r}r!} \rightarrow
\left(
 \frac{1}{2} \right)^{r},\quad\mbox{as\ }n \rightarrow \infty. 
\end{eqnarray*}
\end{proof} 
\begin{corollary}
\label{BC}
The $r$th factorial moment of the number $B_{n}$ of $\b$'s on the second main diagonal of a random staircase tableau of size $n$ satisfies:
\begin{equation}
\mathbb{E}(B_{n})_{r}\to\left(\frac12\right)^r\quad\mbox{as\ } n\to\infty.
\end{equation}
Furthermore,
\begin{equation}
B_{n} \stackrel{d}{\rightarrow} Pois \left( \frac{1}{2} \right)\quad\mbox{as\ } n\to\infty.
\end{equation}\end{corollary}
\begin{proof}
This follows by symmetry, see Section~\ref{DEFST}.
\end{proof}
\begin{remark}
Theorem~\ref{DT} and Corollary~\ref{BC} hold regardless of the values
of $\a$ and $\b$ including the cases discussed earlier when
$\a=\infty$, $\b=\infty$, or $\a=\b=\infty$. As noted in
\cite[Examples~3.6 and 3.7]{HJ} these cases correspond to staircase
tableaux with the maximal number of $\a$'s (or $\b$'s) and the maximal
number of symbols, respectively. The same applies to Theorems~ \ref{thm:symb}, \ref{rsymbols}, and \ref{thm:3rdx} below.
\end{remark}

\subsection{Distribution of non-empty boxes on the second main diagonal}
Random variables  $A_n$ and  $B_n$ 
are not independent random variables, and the second main diagonal may
have empty boxes. Therefore, in order to completely describe the
second main diagonal, we must consider both symbols collectively. First, we present the distribution of non-empty boxes along the second main diagonal. We let $x_j$ denote the event that the box $(n-j,j)$ in the $j$th column and on the second main diagonal in non--empty. 
\begin{theorem}
\label{rNZero}
Let $ 1 \leq j_{1} < ... < j_{r} \leq n - 1$. If (\ref{2_diff}) holds 
then 
\[\mathbb{P}_{n, \a, \b}(x_{j_1}, \dots , x_{j_r})  = \prod_{k=1}^{r}\frac{1}{n + a + b - r + k - 1}.\]
(For $r=1$, this is obtained by adding the expressions in (\ref{1BOXA}).)
Otherwise, 
\[\mathbb{P}_{n, \a, \b}(x_{j_1} , \dots , x_{j_r})  = 0.\] 
\end{theorem}
\begin{proof}
Suppose (\ref{2_diff}) holds. We proceed by induction on $r$. As in the proof of Theorem~\ref{rSymbols}  by passing to $\hn:=n-(j_1-1)$, $\hj_i:=j_i-(j_1-1)$, and $\hbb:=b+j_1-1$ we may assume that $j_1=1$.

By the law of total probability we have
\begin{eqnarray*} 
\mathbb{P}_{n , \a, \b}(x_1, \dots , x_{ j_r}) 
&=&
\mathbb{P}_{n , \a, \b}(x_{j_2}, \dots ,x_{j_r}  \hspace{1mm} |
\hspace{1mm} x_1 = \a)
\mathbb{P}_{n , \a, \b} (x_1 = \a)\\&&\quad+
\mathbb{P}_{n , \a, \b}(x_{j_2} , \dots,x_{j_{r}} \hspace{1mm} | \hspace{1mm} x_{1} = \b) 
\mathbb{P}_{n , \a, \b} (x_1 = \b).
\end{eqnarray*}
Now consider two cases: \\
\textbf{Case 1:} $x_1 = \a$.
By Lemma \ref{SWCorner} and the induction hypothesis:
\begin{eqnarray*} &&
\mathbb{P}_{n , \a, \b}(x_{j_2} , \dots , x_{j_r}  \hspace{1mm} | \hspace{1mm} x_1 = \a) 
= 
\mathbb{P}_{n - 2, \a, \b}(x_{j_{2}-2} , \dots , x_{j_r - 2})  \\ &&\quad= \prod_{k=1}^{r-1}\frac{1}{ n -2 + a + {b} - (r-1) + k - 1}
= \prod_{k=1}^{r-1}\frac{1}{ n  + a + {b} - r + k - 2}
\end{eqnarray*}
and by (\ref{1BOXA}), 
\[
\mathbb{P}_{n , \a, \b} (x_1 = \a) = \frac{{b}}{(n + a + b-1)_2
}.
\]
Therefore,  
\begin{eqnarray*} &&
\mathbb{P}_{n , \a, \b}(\a_1, x_{j_2} , \dots,x_{j_r} ) 
=
\frac{{b}}{(n  + a + {b}-1)_2
} \cdot \prod_{k=1}^{r-1}\frac{1}{n + a + {b} - r + k - 2}.
\end{eqnarray*}

 \textbf{Case 2:} $x_1 = \b$. By Lemma~\ref{BSWCorner} 
\begin{eqnarray*} &&
 \mathbb{P}_{n , \a, \b}(x_{j_2}, \dots , x_{ j_r}  \hspace{1mm} | \hspace{1mm} x_{j_1}  = \b) 
 =
\mathbb{P}_{n-1, \a, \b}(x_{j_2-1} , \dots , x_{j_r-1}  \hspace{1mm} | \hspace{1mm} \b_{n-1,1}) \\ &&\quad= 
\frac
{\mathbb{P}_{n-1, \a, \b}(x_{j_2-1} , \dots , x_{j_{r}-1 }, \b_{n-1,1})}{\mathbb{P}_{n-1, \a,
    \b}(\b_{n-1, 1})}. 
\end{eqnarray*}
The numerator is equal to 
\begin{eqnarray}&&\mathbb{P}_{n-1, \a, \b}(x_{j_2-1} , \dots , x_{j_r-1}) \nonumber
- 
\mathbb{P}_{n-1, \a, \b}(x_{j_2-1} , \dots , x_{j_r-1}, \hspace{1mm}  \a_{n-1, 1})  \\  \label{subtract}&&\quad=
\mathbb{P}_{n-1, \a, \b}(x_{j_{2} - 1}, \dots , x_{j_{r} - 1}) \\ \nonumber&&\qquad-
\mathbb{P}_{n-1, \a, \b}(x_{j_2 -1}, \dots , x_{j_r - 1} \hspace{1mm} | \hspace{1mm} \a_{n-1, 1}) 
\mathbb{P}_{n-1, \a, \b} (\a_{n-1, 1} ).
\end{eqnarray}
By \cite[Lemma~7.5]{HJ}  and the induction hypothesis 
the conditional probability above is
\begin{equation} \label{4b}
\mathbb{P}_{n-2, \a, \b} (x_{j_2-2} , \dots , x_{j_r- 2}) = \prod_{k=1}^{r-1} \frac{1}{n  + a + {b} - r + k - 2}. 
\end{equation}
By (\ref{1BOXA}), (\ref{1BOXD}), and the induction
hypothesis, 
\begin{equation} \label{1b}
\mathbb{P}_{n , \a, \b} (x_1 = \b) = \frac{n  + a - 2}{
(n  + a + {b} - 1)_2}
\end{equation}
\begin{equation} \label{2b}
\frac{1}{\mathbb{P}_{n-1, \a, \b}(\b_{n-1, 1})}  = \frac{n  + a + {b} - 2}{n  + a - 2} 
\end{equation}
\begin{equation} \label{3b}
\mathbb{P}_{n-1, \a, \b}(x_{j_2-1}, \dots , x_{j_r -1}) = \prod_{k=1}^{r-1}
\frac{1}{n -1 + a + {b} - r + k 
} 
\end{equation}
\begin{equation}\label{5b}
\mathbb{P}_{n-1, \a, \b} (\a_{n-1, 1}) = \frac{{b}}{n  + a + {b} - 2}.
\end{equation}
Combining (\ref{4b}) - (\ref{5b}): 
\begin{eqnarray*} &&
\mathbb{P}_{n , \a, \b}(\b_1, x_{j_2},\dots , x_{j_r} ) 
=
\frac{1}{n + a + {b}-1} \cdot \Big( \prod_{k=1}^{r-1}
\frac{1}{n - 1 + a + {b} - r + k 
} \\ &&\qquad-
\frac{{b}}{n  + a + {b} - 2
} \prod_{k=1}^{r-1} \frac{1}{n  + a + {b} - r + k - 2} \Big).
\end{eqnarray*}
Adding Case 1 and Case 2: 
\begin{eqnarray*} 
\mathbb{P}_{n, \a, \b}(x_{1}, ... , x_{j_r})
 &=&
\frac{1}{n  + a + {b}-1} \cdot \prod_{k=1}^{r-1} \frac{1}{n  + a + {b} - r + k - 1
} \\ &&\quad=
\prod_{k=1}^{r} \frac{1}{n + a + b - r + k - 1}
\end{eqnarray*}
which proves our assertion  when (\ref{2_diff}) holds.

If  there exists $j_{k}$ such that $j_{k} = j_{k+1} - 1$, then
$\{x_{j_1} ,\dots  , x_{j_r}\}$ implies  that two boxes side by side on the $(n-i, i)$ diagonal are non-empty, which is impossible by the rules of a  staircase tableau. Therefore the probability is $0$.
\end{proof}

As our final result of this section, we consider the number of symbols on the second
main diagonal, which we denote by $X_{n}$. Then $X_{n} =
\sum^{n-1}_{j=1} I_{x_{j}}$ and we clearly have $X_n=A_n+B_n$. The asymptotic distribution of the number of symbols on the second main diagonal is given in the following theorem. It suggests that, even though $A_n$ and $B_n$ are not independent, they should be asymptotically independent, as the limiting law of $X_n$  is the same as the law of the sum of two iid  $\operatorname{Pois}(1/2)$ random variables.
\begin{theorem}\label{thm:symb} As $n\to\infty$, 
\begin{equation} \label{nzp2}
X_{n} \stackrel{d}{\rightarrow} Pois \left( 1 \right).
\end{equation}
\end{theorem}

\begin{proof}
By Theorem~\ref{rNZero}  and Lemma~\ref{fac_mom} 
\begin{eqnarray*}  
\mathbb{E}(X_{n})_{r} 
&=& r!|J_{r,n-1}|\prod_{k=1}^{r}\frac{1}{n + a + b - r + k -1}
\\ &=&r!\left( {n-1\choose r}+O(n^{r-1})\right) \prod_{k=1}^{r}\frac{1}{n + a + b - r + k -1}\\&
\approx& 
r! \cdot \frac{(n-1)^{r}}{r!n^{r}} 
\rightarrow 1 \quad \mbox{as\ } n \rightarrow \infty.
\end{eqnarray*}
The result   follows by \cite[Theorem~20]{B}, as discussed in the proof of Theorem~\ref{DT}.
\end{proof}
\section{The distribution of parameters on the third main diagonal}

For the considerations of the third main diagonal, note that if the box $(n-k-1,k)$ is non-empty, then by the rules of staircase tableaux, the corresponding diagonal boxes $(n-k+1,k)$ and $(n-k-1,k+2)$ are non-empty (in fact, they contain  $\beta$ and $\alpha$, respectively). The next lemma shows that the remaining two boxes $(n-k-1,k+1)$ and $(n-k,k)$ are likely to be empty. We write $x_{i,j}$ to indicate that the $(i,j)$ box of a tableau is non--empty.

\begin{lemma}\label{ndsx}
Consider arbitrary $S\in\overline\mcS_{n}$
 such that  $x_{n-k-1,k}$.  If $x_{n-k-1,k+1}$ or $x_{n-k,k}$, then $\mbP_{n,\a,\b}(S)=O\left(\frac{1}{(n+a+b)^2}\right)$.
\end{lemma}
\begin{proof}  Note that only one of the boxes $(n-k-1,k+1)$, $(n-k,k)$ of $S$ may contain a non--zero symbol, and   in either case, under the assumptions of the lemma,  the subtableau  $S[n-k-1,k]$ has two $\alpha$'s, two $\beta$'s, one empty box, and a symbol corresponding to position $(n-k-1,k)$ in $S$.  By \eqref{SUBT}, \[\mbP_{n,\a,\b}(S)=\mbP_{3,\ha,\hb}(S[n-k-1,k]),\]
 with $\haa=a+n-k-2$ and $\hbb=b+k-1$, and by the above observation
\[ \mbP_{3,\ha,\hb}(S[n-k-1,k])=2\frac{(\ha+\hb)(\ha\hb)^2}{Z_3(\ha,\hb)}=\frac2{(\haa+\hbb+1)^{\bar{2}}}=O\left(\frac{1}{(n+a+b)^2}\right).\]
\end{proof}

\subsection{The asymptotic distribution of $\a$'s on the third main diagonal}
We begin by deriving the asymptotic probability for the joint distribution of $\alpha$'s on the third main diagonal. We let 
$\a_j^{(3)}$ be the event  that there is an $\alpha$ in the $j$th column on the third main diagonal (i.e in the box $(n-j-1,j)$, $1\le j\le n-2$) but since we will be dealing exclusively with the third main diagonal through the remainder of this section we will drop the superscript. 

We can now give the asymptotic joint distribution of the $\alpha$'s on the third main diagonal. 
\begin{lemma}
\label{rsymbols}
Let $ 1 \leq j_{1} < ... < j_{r} \leq n - 2$. If 
\begin{equation}\label{3_diff}j_{l} \leq j_{l + 1} - 3, \hspace{2mm} \forall l = 1, 2, ..., r -
1\end{equation} 
then 
\[\mbP_{n,\a,\b}(\a_{j_1},...,\a_{j_r}) = \prod_{l=1}^{r}\frac{b + j_{r - l + 1} - 2r + 2l - 1}{(n+a+b-2r+2l-1)_2} + O\left(\frac{1}{(n+a+b)^{r+1}}\right). \]
Otherwise,
\[\mbP_{n,\a,\b}(\a_{j_1},...,\a_{j_r}) = O\left(\frac{1}{(n+a+b)^r}\right).
\]
\end{lemma}

\begin{proof} 
The proof is by induction on $r$. 
When $r=1$, 
\beq*
\mbP_{n,\a,\b}(\a_{j_1})&=&
\mbP_{n,\a,\b}(\a_{j_1},\a_{n-j_1-1,j_1+1})
+\mbP_{n,\a,\b}(\a_{j_1},\b_{n-j_1,j_1})
\\
&&\quad+\mbP_{n,\a,\b}(\a_{n-j_1-1,j_1},0_{n-j_1-1,j_1+1},0_{n-j_1,j_1}).
\eeq*
By Lemma \ref{ndsx}, each of the first two probabilities on the right--hand side is $O(1/n^2)$. By \eqref{SUBT} (applied with $i=n-j_1-1$ and $j=j_1$ so that $\haa=a+n-j_1-2$ and $\hbb=b+j_1-1$) and a direct computation for the resulting tableaux of size three, the last probability is
\[\frac{\ha^3\hb}{Z_3(\ha,\hb)}+\frac{\ha^2\hb^2}{Z_3(\ha,\hb)}=\frac{\ha^2\hb}{(\ha\hb)^2(\haa+\hbb+1)^{\overline2}}=
\frac{b+j_1-1}{
(n+a+b-1)_2},
\]
as required.

Assume the statement holds for integers up to $r-1$. As in the earlier proofs assume without loss that $j_1=1$ and consider 
the following  two cases.
\begin{case}\label{first far} $j_2 \geq j_1+3=4$.
\end{case}
\begin{eqnarray}
&&\mbP_{n,\a,\b}(\a_{1},...,\a_{j_r})=\mbP_{n,\a,\b}(\a_{n-2,2},\a_{1},...,\a_{j_r})\nonumber
+
\mbP_{n,\a,\b}(\b_{n-1,1},\a_{1},\dots,\a_{j_r})
\label{3cases}\\&&\quad
+\mbP_{n,\a,\b}(0_{n-1,1},0_{n-2,2}, \a_{1},...,\a_{j_r}).
\end{eqnarray}
The main contribution is the last probability which can be written as follows,
\beq*
\mbP_{n,\a,\b}(\a_{j_2},...,\a_{j_r}|\a_{1},0_{n-1,1},0_{n-2,2})
\mbP_{n,\a,\b}(\a_{1},0_{n-1,1},0_{n-2,2}).
\eeq*
The conditional probability is equal to $\mbP_{n-2,\a,\b}(\a_{j_2-2},...,\a_{j_r-2})$ (this is because the condition forces zeroes in the first and the third column of $S[j_1-1,1]$ above its $(n-2)$nd row and the remaining $n-2$ columns are unrestricted and thus form a general tableau of size $n-2$).    
To compute $\mbP_{n,\a,\b}(\a_{1},0_{n-1,1},0_{n-2,2})$,
remove the top $n-3$ rows setting $\haa:=a+n-3$ and directly calculate that 
this probability is 
\[
(\ha+\b)\frac{\ha^2\b}{Z_3(\ha,\b)}=\frac{b}{(\haa+b+1)^{\overline{2}}}=\frac{b}{(n+a+b-1)_2}.
\]
So then by the induction hypothesis, if $j_l\leq j_{l+1}-3$  for all $l=2,\dots,r-1$ then the last probability in \eqref{3cases} is 
\beq*
&&\left(\prod_{k=1}^{r-1}\frac{b + j_{r - k + 1}  - 2r + 2k-1 }{(n+a+b-2r+2k-1)_2}+O\left(\frac{1}{(n+a+b)^{r}}\right)\right)\frac{b}{(n+a+b-1)_{2}}\\
&&\qquad=\prod_{k=1}^{r}\frac{b + j_{r - k + 1} - 2r + 2k-1}{(n+a+b-2r+2k-1)_2}+O\left(\frac{1}{(n+a+b)^{r+1}}\right).
\eeq*
(Note that in view of \eqref{hats} the last  equality holds regardless of whether $j_1=1$ or $j_1>1$.)
On the other hand, if for some $2\le l\le r-1$, $j_l>j_{l+1}-3$ then by the induction hypothesis again and  \eqref{hats} (which implies that $b=O(n)$) the same expression is 
\beq*
&&O\left(\frac{1}{(n+a+b)^{r-1}}\right)\frac{b}{(n+a+b-1)_2}
=O\left(\frac{1}{(n+a+b)^r}\right).
\eeq*
Now, returning to Equation (\ref{3cases}), the first two probabilities can be calculated as follows,
\beq*
&=&\mbP_{n,\a,\b}(\a_{j_2},...,\a_{j_r}|\a_{n-2,2},\a_{1})\mbP_{n,\a,\b}(\a_{n-2,2},\a_{1})\\
&&+\mbP_{n,\a,\b}(\a_{j_2},...,\a_{j_r}|\b_{n-1,1},0_{n-2,2},\a_{1})\mbP_{n,\a,\b}(\b_{n-1,1},0_{n-2,2},\a_{1})\\
&=&\mbP_{n-3,\a,\b}(\a_{j_2-3},...,\a_{j_r-3})\mbP_{n,\a,\b}(\a_{n-2,2},\a_{1})\\
&&+\mbP_{n-3,\a,\b}(\a_{j_2-3},...,\a_{j_r-3})\mbP_{n,\a,\b}(\b_{n-1,1},0_{n-2,2},\a_{1}).
\eeq*
By the inductive hypothesis
\[\mbP_{n-3,\a,\b}(\a_{j_2-3},...,\a_{j_r-3})=O\left(\frac1{(n+a+b)^{r-1}}\right)=O\left(\frac1{(n+a+b)^{r-1}}\right)
\]
and by
Lemma \ref{ndsx}, each of the other two probabilities above is
\[O\left(\frac{1}{(n+a+b)^2}\right)
.\]
Hence, this probability is 
\[O\left(\frac{1}{(n+a+b)^{r-1}}\right)O\left(\frac{1}{(n+a+b)^2}\right)=O\left(\frac{1}{(n+a+b)^{r+1}}\right).
\]
\begin{case}\label{first close}  $j_2<4$. 
\end{case}
The case where $j_2=2$ is impossible by the rules of staircase tableaux and hence has probability zero. 
 For the case where $j_2=3$
\beq*
\mbP_{n,\a,\b}(\a_{j_1},...,\a_{j_r})
&=&\mbP_{n,\a,\b}(\a_{j_3},...,\a_{j_r})|\a_{1},\a_{2})\mbP_{n,\a,\b}(\a_{1},\a_{2})\\
&=&\mbP_{n-4,\a,\b}(\a_{j_3-4},...,\a_{j_r-4})\mbP_{n,\a,\b}(\a_{1},\a_{2}).
\eeq*
To compute $\mbP_{n,\a,\b}(\a_{1},\a_{2})$, remove the top $n-4$ rows
setting $\haa:=a+n-4$ and directly calculate that 
\begin{equation}\label{2zeros}
\mbP_{n,\a,\b}(\a_{1},\a_{2})=\frac{b^2}{(\haa+b)^{\overline{4}}}=O\left(\frac{b^2}{(n+a+b)^4}\right)=O\left(\frac1{(n+a+b)^2}\right).
\end{equation}
By the induction hypothesis, 
\[\mbP_{n-4,\a,\b}(\a_{j_3-4},...,\a_{j_r-4})=O\left(\frac1{(n+a+b)^{r-2}}\right)=O\left(\frac1{(n+a+b)^{r-2}}\right).\]
Therefore,
\[
\mbP_{n,\a,\b}(\a_{1},...,\a_{j_r})=
O\left(\frac{1}{(n+a+b)^{r-2}}\right)O\left(\frac{1}{(n+a+b)^2}\right)=O\left(\frac{1}{(n+a+b)^r}\right).
\]
Combining Cases \ref{first far} and \ref{first close}, if $j_{l} \leq j_{l + 1} - 3, \hspace{2mm} \forall\  l = 1, 2, \dots, r -1$, then 
\beq*
\mbP_{n,\a,\b}(\a_{j_1},...,\a_{j_r}) 
&=&\prod_{l=1}^{r}\frac{b + j_{r - l + 1} - 2r + 2l - 1}{(n+a+b-2r+2l-1)_2} + O\left(\frac{1}{(n+a+b)^{r+1}}\right).
\eeq*
Otherwise, 
\[
\mbP_{n,\a,\b}(\a_{j_1},...,\a_{j_r}) = 
O\left(\frac{1}{(n+a+b)^r}\right).
\]
\end{proof}

\n The above lemma gives the following result:
\begin{theorem}\label{alpha_on_3rd}
Let $A_{n}^{(3)}$ to be the number of $\alpha$'s on the third main diagonal of a random weighted staircase tableau. Then,
as $n\to\infty$,
\[
A_{n}^{(3)} \stackrel{d}{\rightarrow} Pois \left( \frac{1}{2} \right).
\]
\end{theorem}

\begin{proof}
Write
\[
\sum_{1 \leq j_{1} < ... < j_{r} \leq n - 2} \mathbb{P}(\a_{j_{1}},\ldots, \a_{j_{r}}) 
=\sum_{J} \mathbb{P}(\a_{j_{1}},\ldots, \a_{j_{r}}) +
\sum_{J^c}\mathbb{P}(\a_{j_{1}},\ldots, \a_{j_{r}}), 
\]
where
\[J=\{1\le j_1<\dots<j_r\le n-2:\  \forall\ l=1,\dots,r-1; j_{l+1}-j_l\geq 3\}
\]
and
\[J^c=\{1\le j_1<\dots<j_r\le n-2:\  \exists\ l=1,\dots,r-1; j_{l+1}-j_l<3\}.
\] 
By Lemma~\ref{rsymbols}, 
\beq* 
&&\sum_{J} \mathbb{P}(\a_{j_{1}},\ldots, \a_{j_{r}}) =
\sum_{J}\left(\prod_{l=1}^{r} \frac{b+j_{r-l+1}-2(r-l)-1}{(n+a+b-2(r-l)-1)_2}+O\left(\frac{1}{(n+a+b)^{r+1}}\right)\right)\\
&&\quad\approx \sum_{J}\prod_{l=1}^{r} \frac{j_{l}}{n^2}+{n-2\choose r}\cdot O\left(\frac{1}{n^{r+1}}\right)
=\sum_{J_{r,n-2}}\prod_{l=1}^{r} \frac{j_{l}}{n^2}-\sum_{J_{r,n-2}\setminus J}\prod_{l=1}^{r} \frac{j_{l}}{n^2}+
O\left(\frac{1}{n}\right).
\eeq*
If $(j_1,\dots,j_r)\in J_{r,n-2}\setminus J$ then there exists an $l$ such that $j_{l+1}-j_l=2$ and thus this set has $O\left({n-2\choose r-1}\right)$ elements. Therefore, by Lemma~\ref{LA} the expression above is asymptotic to
\[\sum_{J_{r,n-2}}\prod_{l=1}^{r} \frac{j_{l}}{n^2}+O\left({n-2\choose r-1}\cdot \frac{n^r}{n^{2r}}\right)+O\left(\frac{1}{n}
\right)=
\frac{1}{2^{r}r!}+O\left(\frac{1}{n}\right),\quad\mbox{as\ }n \rightarrow \infty. 
\]
Finally, by Lemma~\ref{rsymbols},
\[\sum_{J^c}\mathbb{P}(\a_{j_{1}},\ldots, \a_{j_{r}})=O\left({n-2\choose r-1}\frac1{n^r}\right)=O\left(\frac1n\right).
\]
Combining these expressions  with Lemma~\ref{fac_mom} completes the proof.
\end{proof}

\subsection{The asymptotic distribution of symbols on the third main diagonal}
In this section we prove that the total number of  symbols on the third main diagonal is asymptotically Poisson with parameter 1. To this end we will prove an analog of Theorem~\ref{rNZero} and Lemma~\ref{rsymbols}. Throughout this section $x_j=x_j^{(3)}$ indicates the event that the box $(n-j-1,j)$ on the third main diagonal and in the $j$th column is non--empty.   We need a preparatory observation. 
\begin{lemma}\label{switchbeta}
If $j_{1}\geq 3$, then
\[
\mbP_{n,\a,\b}(x_{j_1},...,x_{j_r}, 0_{n-1,1},\b_{n,1})=\mbP_{n,\a,\b}(x_{j_1},...,x_{j_r}, \b_{n-1,2}).
\]
\end{lemma}

\begin{proof}
Consider
\beq*
\mbP_{n,\a,\b}(x_{j_1},...,x_{j_r}, 0_{n-1,1},\b_{n,1})&=&\mbP_{n,\a,\b}(x_{j_1},...,x_{j_r}, 0_{n-1,1},\b_{n,1},\a_{n-1,2})\\
&&+\mbP_{n,\a,\b}(x_{j_1},...,x_{j_r}, 0_{n-1,1},\b_{n,1},\b_{n-1,2}).
\eeq*
Notice that 
\[\mbP_{n,\a,\b}(x_{j_1},...,x_{j_r}, 0_{n-1,1},\b_{n,1},\a_{n-1,2})=\mbP_{n,\a,\b}(x_{j_1},...,x_{j_r}, \a_{n,1},\b_{n-1,2})
\]
since the second column is empty above $\a$ and thus does not restrict the first column, except for box $(n-1,2)$ which is empty. Therefore if $\a_{n-1,2}$ and $\b_{n,1}$ are switched, and the first $n-2$ boxes in column one are switched with the first $n-2$ boxes in column two, the weight does not change. Therefore,
\beq*
&&\mbP_{n,\a,\b}(x_{j_1},...,x_{j_r}, 0_{n-1,1},\b_{n,1})\\&&\qquad
=\mbP_{n,\a,\b}(x_{j_1},...,x_{j_r}, \a_{n,1},\b_{n-1,2})+\mbP_{n,\a,\b}(x_{j_1},...,x_{j_r}, \b_{n,1},\b_{n-1,2})\\
&&\qquad=\mbP_{n,\a,\b}(x_{j_1},...,x_{j_r}, \b_{n-1,2}).
\eeq*
\end{proof}

%
%
The following gives the asymptotic joint distribution of non--zero symbols on the third main diagonal.
\begin{lemma}
\label{rsymbolsx}
Let $ 1 \leq j_{1} < ... < j_{r} \leq n - 2$. If 
\begin{equation}\label{4_diff}j_{l} \leq j_{l + 1} - 3, \hspace{2mm} \forall l = 1, 2, ..., r -
1\end{equation} 
then 
\[\mbP_{n,\a,\b}(x_{j_1},...,x_{j_r}) = \prod_{l=1}^{r}\frac{1}{n+a+b-r+l-1} + O\left(\frac{1}{(n+a+b)^{r+1}}\right).\]
Otherwise,
\[\mbP_{n,\a,\b}(x_{j_1},...,x_{j_r}) = O\left(\frac{1}{(n+a+b)^r}\right).
\]
\end{lemma}

\begin{proof} 
The proof is by induction on $r$. When $r=1$ then by  the same argument as in the beginning of the proof of Lemma~\ref{rsymbols}, we get
\beq*
\mbP_{n,\a,\b}(x_{j_1})\ne0)
&=&\frac{n+a+b-3}{(n+a+b-1)_2}+O\left(\frac{1}{(n+a+b)^2}\right)\\
&=&\frac1{n+a+b-1}+O\left(\frac{1}{(n+a+b)^2}\right)
\eeq*
Assume the statement holds for integers up to $r-1$. We may and do assume that $j_1=1$. 
We again consider two cases. 
\setcounter{case}{0}
\begin{case}\label{firstfarx} $j_2 \ge 4$.
\end{case}
As in  \eqref{3cases} we split the  probability into three pieces according to 
whether  the boxes $(n-1,1)$, $(n-2,2)$ are empty or  not 
\begin{eqnarray}
\mbP_{n,\a,\b}(x_{1},...,x_{j_r})
&=&\mbP_{n,\a,\b}(0_{n-1,1},0_{n-1-1,2}, x_{1},...,x_{j_r})\nonumber\\&&\quad
+\mbP_{n,\a,\b}(\b_{n-1,1},0_{n-2,2},x_{1},...,x_{j_r})\label{3casesx}\\&&\quad
+\mbP_{n,\a,\b}(\a_{n-2,2},x_{1},...,x_{j_r}).\nonumber
\end{eqnarray}
The main contribution is the first probability which can be broken up into two cases.
%
\begin{subcase}\label{alpha} $x_1=\a$.
\end{subcase}
This is the case considered in the proof of Lemma~\ref{rsymbols}  and gives 
\beq*
&&\mbP_{n,\a,\b}(0_{n-1,1},0_{n-2,2}, \a_{1},x_{j_2},...,x_{j_r})\\
&&\qquad=\mbP_{n-2,\a,\b}(x_{j_2-2},...,x_{j_r-2})\mbP_{n,\a,\b}(0_{n-1,1},0_{n-2,2}, \a_{1}).
\eeq*

\begin{subcase}\label{beta}$x_1=\b$.
\end{subcase}
Then
\beq*
&&\mbP_{n,\a,\b}(0_{n-1,1},0_{n-2,2}, \b_{1},x_{j_2},...,x_{j_r})\\
&&\quad=\mbP_{n,\a,\b}(x_{j_2},...,x_{j_r}\hspace{1mm}|\hspace{1mm}0_{n-1,1},0_{n-2,2}, \b_{1})\mbP_{n,\a,\b}(0_{n-1,1},0_{n-2,2}, \b_{1}).\eeq*
As can be seen by removing the $(n-2)$nd row and the third column, the conditional probability is equal to 
\beq*&&\mbP_{n-1,\a,\b}(x_{j_2-1},...,x_{j_r-1}\hspace{1mm}|\hspace{1mm}0_{n-2,1}, \b_{n-1,1})=
\mbP_{n-1,\a,\b}(x_{j_2-1},...,x_{j_r-1}\hspace{1mm}|\hspace{1mm}\b_{n-2,2})\\
&&\qquad=\mbP_{n-2,\a,\widetilde\b}(x_{j_2-2},...,x_{j_r-2}\hspace{1mm}|\hspace{1mm}\b_{n-2,1}),\qquad\mbox{with $\widetilde b=b+1$,}
\eeq*
where the first equality  above follows  by writing the probability of the left as the ratio and applying Lemma~\ref{switchbeta} to both numerator and denominator and the second follows from \eqref{SUBT} by observing that at this point no entries of the first column are involved so that it can be removed.

At this point $\b_{n-2,1}$ is a corner box of staircase tableau of size $n-2$ and we can use \eqref{subtract} to get 
\beq*
&&\mbP_{n-2,\a,\widetilde\b}(x_{j_2-2},...,x_{j_r-2},\b_{n-2,1})\\&&\qquad=
\mbP_{n-2,\a,\widetilde\b}(x_{j_2-2},...,x_{j_r-2})-\mbP_{n-3,\a,\widetilde\b}(x_{j_2-3},\dots,x_{j_r-3})\mbP_{n-2,\a,\widetilde\b}(\a_{n-2,1})
.
\eeq*
Combining the above expressions and adding Case~\ref{alpha} and Case~\ref{beta} we obtain 
\begin{eqnarray}\label{mainprob}
\nonumber&&\mbP_{n,\a,\b}(x_1,\dots,x_{j_r})=\mbP_{n-2,\a,\b}(x_{j_2-2},...,x_{j_r-2})\mbP_{n,\a,\b}(0_{n-1,1},0_{n-2,2}, \a_{1})\\&&\quad+\mbP_{n-2,\a,\widetilde\b}(x_{j_2-2},...,x_{j_r-2})\frac{\mbP_{n,\a,\b}(0_{n-1,1},0_{n-2,2}, \b_{1})}{\mbP_{n-2,\a,\widetilde\b}(\b_{n-2,1})}\\&&\quad-\mbP_{n-3,\a,\widetilde\b}(x_{j_2-3},...,x_{j_r-3})\frac{\mbP_{n,\a,\b}(0_{n-1,1},0_{n-2,2}, \b_{1})\mbP_{n-2,\a,\widetilde\b}(\a_{n-2,1})}{\mbP_{n-2,\a,\widetilde\b}(\b_{n-2,1})}.\nonumber
\end{eqnarray}
%
%
%
%
Note that by \eqref{SUBT} 
\begin{equation}\label{first}
\mbP_{n-2,\a,\b}(x_{j_2-2},...,x_{j_r-2})=\mbP_{n-3,\a,\widetilde\b}(x_{j_2-3},...,x_{j_r-3}).
\end{equation}
Now, using \eqref{SUBT} with $\haa=n+a-3$, we calculate directly,
\begin{eqnarray}
\mbP_{n,\a,\b}(0_{n-1,1},0_{n-2,2}, \a_{1})&=&\mbP_{3,\ha,\b}(0_{2,1},0_{1,2}, \a_{1})=\frac{b}{(n+a+b-1)_2}\\
\mbP_{n,\a,\b}(0_{n-1,1},0_{n-2,2}, \b_{1})
&=&\frac{n+a-3}{(n+a+b-1)_2}.
\end{eqnarray}
%
%
Also,  by \eqref{1BOXD} 
\begin{eqnarray}
\mbP_{n-2,\a,\widetilde\b}(\a_{n-2,1})&=&\frac{\widetilde b }{n-2+a+\widetilde b-1}=\frac{b+1}{n+a+b-2}, \\
\mbP_{n-2,\a,\widetilde\b}(\b_{n-2,1})&=&
\frac{n+a-3}{n+a+b-2}.\label{last}
\end{eqnarray}
So by substituting Equations (\ref{first})-(\ref{last}) into Equation (\ref{mainprob}) and simplifying 
\beq*
&&\mbP_{n,\a,\b}(x_1,\dots,x_{j_r})
=\mbP_{n-2,\a,\widetilde\b}(x_{j_2-2},...,x_{j_r-2})\frac{1}{n+a+b-1}\\
&&\qquad+\mbP_{n-3,\a,\widetilde b}(x_{j_2-3},...,x_{j_r-3})
\cdot\left(\frac{b}{(n+a+b-1)_2}-\frac{b+1}{(n+a+b-1)_2} \right).
\eeq*
The second summand by the induction hypothesis  is 
\[O\left(\frac1{(n+a+b)^{r-1}}\cdot\frac1{(n+a+b)^2}\right)=O\left(\frac1{(n+a+b)^{r+1}}\right).\]
If $j_l\leq j_{l+1}-3$  for all $l=2,\dots,r-1$ then by the induction hypothesis, the first summand is
\beq*
&&\left(\prod_{l=1}^{r-1}\frac{1}{n-2+a+\widetilde b-(r-1)+l-1}+O\left(\frac{1}{(n-2+a+\widetilde b)^r}\right)\right)\frac{1}{n+a+b-1}\\
&&\qquad=\left(\prod_{l=1}^{r-1}\frac{1}{n+a+b-r+l-1}+O\left(\frac{1}{(n+a+b-1)^r}\right)\right)\frac{1}{n+a+b-1}\\
&&\qquad=\prod_{l=1}^r\frac{1}{n+a+b-r+l-1}+O\left(\frac{1}{(n+a+b-1)^{r+1}}\right).
\eeq*
On the other hand, if for some $2\le l\le r-1$, $j_l>j_{l+1}-3$ then by the induction hypothesis again the same expression is,
\beq*
O\left(\frac{1}{(n-2+a+\widetilde b)^{r-1}}\right)\frac{1}{n+a+b-1}
=O\left(\frac{1}{(n+a+b-1)^r}\right).
\eeq*

Now, returning to Equation (\ref{3casesx}), the first two probabilities can be calculated in a similar manner to obtain an equation like Equation (\ref{mainprob}). For example, $\mbP_{n,\a,\b}(\a_{n-2,2},x_{1},...,x_{j_r})$ is
\beq*
&&\mbP_{n-3,\a,\b}(x_{j_2-3},...,x_{j_r-3})\mbP_{n,\a,\b}(\a_{n-2,2}, \a_{1})
\\&&\qquad+\mbP_{n-2,\a,\b}(x_{j_2-2},...,x_{j_r-2}|\b_{n-2,1})\mbP_{n,\a,\b}(\a_{n-2,2},\b_{1})\\
&&\quad=\mbP_{n-3,\a,\b}(x_{j_2-3},...,x_{j_r-3})\mbP_{n,\a,\b}(\a_{n-2,2}, \a_{1})
\\&&\qquad+\mbP_{n-2,\a,\b}(x_{j_2-2},...,x_{j_r-2})\frac{\mbP_{n,\a,\b}(\a_{n-2,2},\b_{1})}{\mbP_{n-2,\a,\b}(\b_{n-2,1})}
\\&&\qquad-\mbP_{n-3,\a,\b}(x_{j_2-2},...,x_{j_r-2})\frac{\mbP_{n,\a,\b}(\a_{n-2,2},\b_{1})\mbP_{n-2,\a,\b}(\a_{n-2,1})}{\mbP_{n-2,\a,\b}(\b_{n-2,1})}.
\eeq*
By \eqref{SUBT} and \eqref{1BOXD}
\[
\frac{\mbP_{n,\a,\b}(\a_{n-2,2},\b_{1})}{\mbP_{n-2,\a,\b}(\b_{n-2,1})}
=\frac{\frac{a+n-3}{(n+a+b-1)_3}}{\frac{n+a-3}{n+a+b-3}}
=O\left(\frac{1}{(n+a+b)^2}\right)
\]
and 
\beq*&&
\mbP_{n,\a,\b}(\a_{n-2,2}, \a_{1})
-\frac{\mbP_{n,\a,\b}(\a_{n-2,2},\b_{1})\mbP_{n-2,\a,\b}(\a_{n-2,1})}{\mbP_{n-2,\a,\b}(\b_{n-2,1})}
\\&&\quad=\frac b{(n+a+b-1)_3}-\frac{\frac{a+n-3}{(n+a+b-1)_3}\cdot\frac b{n+a+b-3}}{\frac{n+a-3}{n+a+b-3}}=0.\eeq*
Therefore, by the induction hypothesis, 
\[\mbP_{n,\a,\b}(\a_{n-2,2},x_{1},...,x_{j_r})=O\left(\frac{1}{(n+a+b)^{r+1}}\right).
\]
In the same way, 
\[
\mbP_{n,\a,\b}(\b_{n-1,1},0_{n-2,2},x_{1},...,x_{j_r})=O\left(\frac{1}{(n+a+b)^{r+1}}\right).
\]
\begin{case}\label{firstclosex}  $j_2 <4$. 
\end{case}
The case where $j_2=3$ is impossible by the rules of staircase tableaux and hence has probability zero. For the case where $j_2=2$,  by the law of total probability
\begin{eqnarray}\label{4_prob}
\mbP_{n,\a,\b}(x_{1},...,x_{j_r})&=&\nonumber
\mbP_{n,\a,\b}(x_{j_3}...,x_{j_r}|\a_{1},\a_{2})\mbP_{n,\a,\b}(\a_{1},\a_{2})\\ \nonumber
&&+\mbP_{n,\a,\b}(x_{j_3},...,x_{j_r}|\a_{1},\b_{2})\mbP_{n,\a,\b}(\a_{1},\b_{2})\\
&&+\mbP_{n,\a,\b}(x_{j_3},...,x_{j_r}|\b_{1},\a_{2})\mbP_{n,\a,\b}(\b_{1},\a_{2})\\
&&+\mbP_{n,\a,\b}(x_{j_3},...,x_{j_r}|\b_{1},\b_{2})\mbP_{n,\a,\b}(\b_{1},\b_{2}).\nonumber
\end{eqnarray}
The first term is handled in Case~2 of Lemma~\ref{rsymbols} and is $O(1/(n+a+b)^r)$. The second, by removing the last three rows of a tableau and its  first, third, and fourth column is
\beq*
&&\mbP_{n-3,\a,\b}(x_{j_3-3},...,x_{j_r-3}\hspace{1mm}|\hspace{1mm}\b_{n-3,1})\mbP_{n,\a,\b}(\a_{1},\b_{2})\\&&\quad\le
\mbP_{n-3,\a,\b}(x_{j_3-3},...,x_{j_r-3})\frac{\mbP_{n,\a,\b}(\a_{1},\b_{2})}{\mbP_{n-3,\a,\b}(\b_{n-3,1})}.
\eeq*
By \eqref{SUBT} and \eqref{1BOXD} and remembering that $b$ may depend on $n$ (see \eqref{hats}),
\[
\frac{\mbP_{n,\a,\b}(\a_{1},\b_{2})}{\mbP_{n-3,\a,\b}(\b_{n-3,1})}=\frac{\frac{(n+a-4)b}{(n+a+b-4)^{\overline4}}}{\frac{n+a-4}{n+a+b-4}}=O\left(\frac b{(n+a+b)^3}\right)=O\left(\frac 1{(n+a+b)^2}\right)
\]
so that, by the induction hypothesis, the whole term is 
\[O\left(\frac1{(n+a+b)^{r-2}}\cdot\frac1{(n+a+b)^2}\right)=O\left(\frac1{(n+a+b)^r}\right).\]
Finally, consider the sum of the last two terms in \eqref{4_prob}. By removing  the second and third column along with the $(n-1)$st and $(n-2)$nd row for the first probability and the third and forth column and the $(n-2)$nd and $(n-3)$rd row for the other they are 
\beq*
&&\mbP_{n-2,\a,\b}(x_{j_3-2},...,x_{j_r-2}\hspace{1mm}|\hspace{1mm}\b_{n-2,1},\a_{n-3,2})\mbP_{n,\a,\b}(\b_{1},\a_{2})\\
&&\qquad+\mbP_{n-2,\a,\b}(x_{j_3-2},...,x_{j_r-2}\hspace{1mm}|\hspace{1mm}\b_{n-2,1},\b_{n-3,2})\mbP_{n,\a,\b}(\b_{1},\b_{2})\\
&&\quad=\mbP_{n-2,\a,\b}(x_{j_3-2},...,x_{j_r-2},\b_{n-2,1},\a_{n-3,2})\frac{\mbP_{n,\a,\b}(\b_{1},\a_{2})}{\mbP_{n-2,\a,\b}(\b_{n-2,1},\a_{n-3,2})}\\
&&\qquad+\mbP_{n-2,\a,\b}(x_{j_3-2},...,x_{j_r-2},\b_{n-2,1},\b_{n-3,2})\frac{\mbP_{n,\a,\b}(\b_{1},\b_{2})}{\mbP_{n-2,\a,\b}(\b_{n-2,1},\b_{n-3,2})}.
\eeq*
We now observe that 
\beq*
\frac{\mbP_{n,\a,\b}(\b_{1},\a_{2})}{\mbP_{n-2,\a,\b}(\b_{n-2,1},\a_{n-3,2})}=\frac{\mbP_{n,\a,\b}(\b_{1},\b_{2})}{\mbP_{n-2,\a,\b}(\b_{n-2,1},\b_{n-3,2})}=\frac1{(n+a+b-2)^{\overline2}}
\eeq*
as can be seen by applying \eqref{SUBT} to each of the four probabilities. 
Further,  since the box $(n-3,2)$ is on the main diagonal of a staircase tableau of size $n-2$, $\b_{n-3,2}$ and $\a_{n-3,2}$ are complements of each other.
Therefore,  the sum of the  last two terms in \eqref{4_prob} is
\beq*
&&\mbP_{n-2,\a,\b}(x_{j_3-2},...,x_{j_r-2},\b_{n-2,1})\frac{1}{(n+a+b-1)_2}\\&&\qquad\le
\mbP_{n-2,\a,\b}(x_{j_3-2},...,x_{j_r-2})\frac{1}{(n+a+b-1)_2}=O\left(\frac1{(n+a+b)^{r}}\right)
\eeq*
by the induction hypothesis. 
\end{proof}

\n We can now give 
the asymptotic distribution of  the number of symbols on the third main diagonal. 
\begin{theorem}\label{thm:3rdx}
Let  $X^{(3)}_{n}$  be the number of symbols on the third main diagonal, i.e. 
$X_{n}^{(3)} := \sum^{n-2}_{j=1} I_{x_{j}}$. Then,
as $n\to\infty$,
\[
X_{n}^{(3)} \stackrel{d}{\rightarrow} Pois \left( 1 \right).
\]
\end{theorem}

\begin{proof}
As in the proof of Theorem~\ref{alpha_on_3rd} we 
split the sum 
\begin{equation}\sum_{1 \leq j_{1} < ... < j_{r} \leq n - 2} \mathbb{P}(x_{j_{1}},\ldots, x_{j_{r}}) =
 \sum_{J} \mathbb{P}(x_{j_{1}},\ldots, x_{j_{r}}) +
\sum_{J^c}\mathbb{P}(x_{j_{1}},\ldots, x_{j_{r}}). 
\label{two_termsx}
\end{equation}
By Lemma~\ref{rsymbolsx}, 
\beq* 
&&\sum_{J} \mathbb{P}(x_{j_{1}},\ldots, x_{j_{r}}) =
\sum_{J}\left(\prod_{k=1}^{r}\frac{1}{n+a+b-r+k-1} + O\left(\frac{1}{(n+a+b)^{r+1}}\right)\right)\\
&&=\left({n-2\choose r}+O\left(n^{r-1}\right)\right)
\left(\prod_{k=1}^{r}\frac{1}{n+a+b-r+k-1} + O\left(\frac{1}{(n+a+b)^{r+1}}\right)\right)\\
&&=
\frac{1}{r!}+O\left(\frac{1}{n}\right),\quad\mbox{as\ }n \rightarrow \infty. 
\eeq*
Finally, since cardinality of $J^c$ is $O\left({n-2\choose r-1}\right)$, by Lemma~\ref{rsymbolsx},
\[\sum_{J^c}\mathbb{P}(x_{j_{1}},\ldots, x_{j_{r}})=O\left({n-2\choose r-1}\frac1{n^r}\right)=O\left(\frac1n\right).
\]
Combining the last two expressions with \eqref{two_termsx} proves that 
\[\mathbb{E}(X_{n})_{r} =r! \left( \sum_{1 \leq j_{1} < ... < j_{r} \leq n - 2} \mathbb{P}(x_{j_{1}},\ldots, x_{j_{r}})\right)\to1,\quad \mbox{as\  } n\to\infty.\]  
\end{proof}


\end{document}